\documentclass[12pt]{amsart}
\pdfoutput=1

\usepackage[colorlinks,linkcolor=blue,citecolor=purple]{hyperref}
\usepackage[obeyFinal,textsize=scriptsize,shadow,loadshadowlibrary]{todonotes}
\usepackage[shortlabels]{enumitem}
\usepackage{amsmath}
\usepackage{amssymb}
\usepackage{amsthm}
\usepackage{booktabs}
\usepackage[nameinlink]{cleveref}
\usepackage{derivative}
\usepackage{graphicx}
\usepackage{mathdots}
\usepackage{mathrsfs}
\usepackage{mathtools}
\usepackage{microtype}

\usepackage{asymptote}
\usepackage{tikz-cd}
\usetikzlibrary{decorations.pathmorphing}

\allowdisplaybreaks

\usepackage[backend=biber,backref=true,style=alphabetic]{biblatex}
\DefineBibliographyStrings{english}{%
  backrefpage  = {cited p.}, 
  backrefpages = {cited pp.} 
}
\DeclareFieldFormat[article]{volume}{\textbf{#1}\addcolon\space} 
\DeclareSourcemap{\maps[datatype=bibtex]{\map{\step[fieldsource=ISBN-13,fieldtarget=ISBN]}}}

\usepackage{thmtools}
\declaretheorem[name=Theorem,numberwithin=section]{theorem}
\declaretheorem[name=Lemma,sibling=theorem]{lemma}
\declaretheorem[name=Proposition,sibling=theorem]{proposition}

\declaretheorem[name=Conjecture,sibling=theorem]{conjecture}

\declaretheorem[name=Definition,sibling=theorem,style=definition]{definition}

\declaretheorem[name=Remark,sibling=theorem,style=definition]{remark}

\providecommand{\wt}{\widetilde}

\providecommand{\CC}{\mathbb C}

\providecommand{\QQ}{\mathbb Q}

\providecommand{\ZZ}{\mathbb Z}

\usepackage{listings}
\newcommand{\NN}{\mathbb{Z}_{\ge 0}}
\usepackage[margin=1in]{geometry}
\usepackage{xcolor}
\usepackage{comment}


\title[Syzygies of numerical semigroups]
{Fel's Conjecture on Syzygies of Numerical Semigroups}
\date{February 3, 2026}
\newif\ifmanyauthors
\manyauthorstrue 

\ifmanyauthors
  \usepackage[foot]{amsaddr}
  \usepackage{textcomp}
  \newcommand{\dmd}{\ensuremath{\diamond}}

  \author{Evan Chen\textsuperscript{\dag}}
  \email{evan@axiommath.ai}

  \author{Chris Cummins\textsuperscript{*}}
  \email{chris@axiommath.ai}

  \author{GSM\textsuperscript{*}}

  \author{Dejan Grubisic\textsuperscript{*}}
  \email{dejan@axiommath.ai}

  \author{Leopold Haller\textsuperscript{*}}
  \email{leo@axiommath.ai}

  \author{Letong Hong\textsuperscript{\dmd}}
  \email{carina@axiommath.ai}

  \author{Andranik Kurghinyan\textsuperscript{*}}
  \email{andranik@axiommath.ai}

  \author{Kenny Lau\textsuperscript{\dag*}}
  \email{kenny@axiommath.ai}

  \author{Hugh Leather\textsuperscript{*}}
  \email{hugh@axiommath.ai}

  \author{Seewoo Lee\textsuperscript{\dag}}
  \email{seewoo@axiommath.ai}

  \author{Aram Markosyan\textsuperscript{*}}
  \email{am@axiommath.ai}

  \author{Ken Ono\textsuperscript{\dag}}
  \email{ken@axiommath.ai}

  \author{Manooshree Patel\textsuperscript{*}}
  \email{manooshree@axiommath.ai}

  \author{Gaurang Pendharkar\textsuperscript{*}}
  \email{gaurang@axiommath.ai}

  \author{Vedant Rathi\textsuperscript{*}}
  \email{vedant@axiommath.ai}

  \author{Alex Schneidman\textsuperscript{*}}
  \email{alex@axiommath.ai}

  \author{Volker Seeker\textsuperscript{*}}
  \email{volker@axiommath.ai}

  \author{Shubho Sengupta\textsuperscript{\dmd}}
  \email{shubho@axiommath.ai}

  \author{Ishan Sinha\textsuperscript{*}}
  \email{ishan@axiommath.ai}

  \author{Jimmy Xin\textsuperscript{*}}
  \email{jimmy@axiommath.ai}

  \author{Jujian Zhang\textsuperscript{\dag*}}
  \email{jujian@axiommath.ai}

  \address{Axiom Math, 124 University Avenue, Palo Alto, CA 94301}

\else
  \usepackage{amsaddr}

  \author{Evan Chen$^{a}$}
  \email{evan@axiommath.ai}
  \thanks{$^{a}$Axiom Math, 124 University Avenue, Palo Alto, CA 94301.
    Corresponding author: \texttt{evan@axiommath.ai}.}

  \author{Kenny Lau$^{a}$}
  \email{kenny@axiommath.ai}

  \author{Ken Ono$^{a}$}
  \email{ken@axiommath.ai}

  \author{Jujian Zhang$^{a}$}
  \email{jujian@axiommath.ai}

  \hypersetup{
    pdfauthor={Evan Chen, Kenny Lau, Ken Ono, and Jujian Zhang},
  }
\fi

\hypersetup{
  pdftitle={On the connection between syzygy genera and a sequence of universal symmetric polynomials}
}

\begin{document}
\maketitle

\ifmanyauthors
\begin{center}
  \footnotesize
  Authors are listed alphabetically. \\
  \textsuperscript{\dag}Mathematical contributor,
  \textsuperscript{*}Engineering contributor,
  \textsuperscript{\dmd}Principal investigator.
\end{center}
\fi

\begin{abstract}
Let $S=\langle d_1,\dots,d_m\rangle$ be a numerical semigroup and $k[S]$ its semigroup ring.
The Hilbert numerator of $k[S]$ determines normalized alternating syzygy power sums $K_p(S)$
encoding alternating power sums of syzygy degrees.
Fel conjectured an explicit formula for $K_p(S)$, for all $p\ge 0$, in terms of the gap power sums
$G_r(S)=\sum_{g\notin S} g^r$ and universal symmetric polynomials $T_n$ evaluated at the generator power sums
$\sigma_k=\sum_i d_i^k$ (and $\delta_k=(\sigma_k-1)/2^k$).
We prove Fel's conjecture via exponential generating functions and coefficient extraction, isolating the
universal identities for $T_n$ needed for the derivation.
The argument is fully formalized in Lean/Mathlib, and was produced automatically by AxiomProver
from a natural-language statement of the conjecture.
\end{abstract}

\section{Introduction}
\label{sec:intro}
Numerical semigroups form a classical meeting point of
additive number theory and commutative algebra.
Throughout this paper, let
\[ S \coloneq \left\langle d_1, \dots, d_m \right\rangle \subseteq \NN \]
denote a semigroup with $m \ge 1$ generators ($\gcd(d_1,\dots,d_m)=1$), and let
\[ \Delta \coloneq \NN \setminus S \]
denote the gap set (finite by definition).
We always assume $0 \in S$.
The semigroup ring
\[ k[S] \cong k[t^{d_1},\dots,t^{d_m}] \]
is a one-dimensional graded domain whose Hilbert series records
fine arithmetic data of $S$ through $\Delta$.
The interplay between gaps, Hilbert series, and graded free resolutions
links numerical semigroups to topics such as Frobenius-type problems,
restricted partition functions, and toric/monomial methods.
We refer the reader to
\cite{ref:RosalesGarciaSanchez} for background on numerical semigroups and
\cite{ref:MillerSturmfels} for commutative-algebraic and combinatorial aspects
of semigroup rings and their resolutions.

In recent work, Fel \cite{ref:fel0,ref:fel} introduced alternating power sum coefficients
built from the degrees of syzygy in the Hilbert numerator of $k[S]$,
and conjectured a uniform closed formula (\Cref{conj:main} below)
expressing certain normalized combinations $K_p(S)$ in terms of gap power sums
\begin{equation}
  G_r(S) \coloneq \sum_{g \in \Delta} g^r
  \label{eq:gap}
\end{equation}
and universal symmetric polynomials $T_n$.
In this paper, we prove Fel's conjecture for all $p\ge 0$
by a direct formal-power-series argument.

\subsection{Hilbert series of numerical semigroups, and alternating syzygy}
We begin by recalling the relevant mathematical objects related to $k[S]$.

\begin{definition}
  \label{def:hilbert}
  The \emph{Hilbert series} for $S$ is defined as the ordinary generating
  function $H_S(z) \coloneq \sum_{s \in S} z^s$.
  From the Hilbert-Serre-Poincar\'{e} theorem \cite[Theorem 11.1]{AM69},
  one may write
  \[ H_S(z) = \frac{Q_S(z)}{\prod_1^m (1-z^{d_i})} \]
  for some polynomial $Q_S(z) \in \ZZ[z]$, called the \emph{Hilbert numerator} of $S$.
\end{definition}
The Hilbert numerator $Q_S$ takes a specific form: it can be written as
an alternating sum of monomials
\begin{equation}
  Q_S(z) = 1 - \sum_{j=1}^{\beta_1} z^{C_{1,j}}
  + \sum_{j=1}^{\beta_2} z^{C_{2,j}} - \sum_{j=1}^{\beta_3} z^{C_{3,j}} + \dots
  + (-1)^{m-1} \sum_{j=1}^{\beta_{m-1}} z^{C_{m-1,j}}
  \label{eq:syzygy}
\end{equation}
where $\beta_i > 0$ are called the \emph{partial Betti numbers}
and the $C_{i,j}$ are called the \emph{degrees of syzygies},
as defined in \cite[\S2]{ref:fel0}.
(We note in \emph{loc.\ cit.}\ that there are additional constraints on $\beta_i$
and $C_{i,j}$, but they are not needed for the subsequent \Cref{def:alternating}.)

The degrees of syzygy allow for the statement of several surprising results.
For example, define the alternating power sums of syzygy degrees as follows.
\begin{definition}
  \label{def:alternating}
  For $r \ge 0$ define the sum of the alternating $r$th powers by:
  \[ \CC_r(S) \coloneq \sum_{i=1}^{m-1} \sum_{j=1}^{\beta_i} (-1)^i C_{i,j}^r. \]
\end{definition}

\begin{definition}
  Additionally, define the shorthand
  \[ \pi_m \coloneq d_1 \dots d_m. \]
\end{definition}
One of the main theorems of \cite{ref:fel0} computes $\CC_r$.
\begin{theorem}
  [{\cite[Theorem 1]{ref:fel0}}]
  \label{thm:Kp}
  We have $\CC_0(S) = 1$, $\CC_r(S) = 0$ for $1 \le r \le m-2$, and
  \[ \CC_{m-1}(S) = (-1)^m (m-1)! \pi _m. \]
\end{theorem}

\subsection{Fel's sequence $T_n$ of universal symmetric polynomials}
What about $\CC_r(S)$ for $r \ge m$?
In \cite{ref:fel}, these $\CC_n(S)$ were reparametrized in terms of
a certain explicit coefficient denoted $K_p(S)$.
\begin{theorem}
  [{\cite[\S3]{ref:fel}}]
  \label{thm:Cmp_Kp}
  There exists a coefficient $K_p(S)$ such that for every $p \ge 0$ we have
  \[ \CC_{m+p}(S) = (-1)^m \pi_m \cdot \frac{(m+p)!}{p!} \cdot K_p(S). \]
  This coefficient $K_p$ is a linear combination of the $G_0$, $G_1$, \dots, $G_p$
  defined by \eqref{eq:gap}.
\end{theorem}
The quantity $G_i$ are called the \emph{genera} or \emph{gap power sums} of $S$.
(In particular, $G_0 = |\Delta|$ is called the genus of $S$.)
In \cite[equation (23)]{ref:fel}, several $K_p$ are established by direct calculation.
Let $\sigma_k \coloneq \sum_1^m d_i^k$ and $\delta_k \coloneq \frac{\sigma_k-1}{2^k}$;
then the formula for $0 \le p \le 3$ are:
\begin{align*}
  K_0 &= G_0 + \delta_1 \\
  K_1 &= G_1 + \tfrac{\sigma_1}{2} G_0 + \tfrac{3\delta_1^2 + \delta_2}{6} \\
  K_2 &= G_2 + \sigma_1 G_1 + \tfrac{3\sigma_1^2 + \sigma_2}{12} G_0
    + \tfrac{\delta_1(\delta_1^2+\delta_2)}{3} \\
  K_3 &= G_3 + \tfrac{3\sigma_1}{2} G_2 + \tfrac{3\sigma_1^2 + \sigma_2}{12} G_1
    + \tfrac{\sigma_1(\sigma_1^2 + \sigma_2)}{8} G_0
    + \tfrac{15\delta_1^4 + 30\delta_1^2\delta_2 + 5\delta_2^2 - 2\delta_4}{60}.
\end{align*}
These formulas for $K_p$ have been written in a suggestive form,
where a special sequence of polynomials
\[ T_0(\sigma) = 1, \quad
  T_1(\sigma) = \frac{\sigma_1}{2}, \quad
  T_2(\sigma) = \frac{3\sigma_1^2+\sigma_2}{12}, \quad
  T_3(\sigma) = \frac{\sigma_1(\sigma_1^2 + \sigma_2)}{8}, \; \dots \]
given by $T_n(\sigma) \coloneq T(\sigma_1, \sigma_2, \dots, \sigma_n)$ seem to reappear.
These polynomials are defined in \cite{ref:felpoly}
and we describe these $T_n$ fully in \Cref{def:A} and \Cref{def:B}.
Throughout, we refer to them as \emph{universal symmetric polynomials}.
Fel conjectured that the $T_n$ describe $K_p(S)$ fully;
analogously writing $T_n(\delta) \coloneq T_n(\delta_1, \delta_2, \dots, \delta_n)$, we have:
\begin{conjecture}
  [{\cite[Conjecture 1]{ref:fel}}]
  \label{conj:main}
  For every $p \ge 0$, we have
  \[ K_p(S) = \sum_{r=0}^p \binom pr T_{p-r}(\sigma) G_r(S)
    + \frac{2^{p+1}}{p+1} T_{p+1}(\delta). \]
\end{conjecture}
Two additional conjectures about $T_n$
(not directly related to semigroups \emph{per se})
are also stated in \cite{ref:felpoly},
which we describe in the next section.

\subsection{The reappearance of $T_n$ in relation to Ramanujan's $U_{2n}(q)$ and other contexts}
This section provides some historical background
showing three appearances of the universal symmetric polynomials $T_n$
outside the context of numerical semigroups, following \cite{ref:quasimodular}.
This theory will not be used in our proof itself of our main theorem,
but establishes the contexts in which $T_n$ appears elsewhere in mathematics.

\subsubsection{Restricted partitions}
For positive integers $d_1$, \dots, $d_m$, let $W(s, d_\bullet)$
denote the restricted partition function
(that is, the number of integer partitions of $s \ge 0$ into $d_i$).
Sylvester \cite{ref:sylvester} decomposes $W(s, d_\bullet)$
as the sum of so-called \emph{Sylvester waves} $W_q$:
\[ W(s, d_\bullet) = \sum_{\exists i : q \mid d_i} W_q(s, d_\bullet).  \]
In particular, the \emph{first Sylvester wave} $W_1$ is of special interest:
it has the form (see \cite[(3.16) and (7.1)]{ref:fel0})
\[ W_1(s, d_\bullet) = \frac{1}{(m-1)!} \sum_{r=0}^{m-1} f_r(d_\bullet) s^{m-1-r}
  \quad f_r(d_\bullet) \coloneq  \left( \sigma_1 + \sum_{i=1}^m \mathcal{B}_i d_i \right)^r, \]
with each $\mathcal{B}_i$ denoting a Bernoulli umbra,
meaning $(\mathcal B_i d_i)^n \coloneq B_n d_i^n$.

However, it turns out that $f_r$ are actually given up to sign changes by
the universal symmetric polynomials $T_n$:
\begin{theorem}
  [{\cite[Conjecture 1]{ref:felpoly}} or {\cite[Theorem 1.5]{ref:quasimodular}}]
  \label{conj:fel1}
  For $n \ge 2$, we have
  \[ f_n(\sigma_1, \dots, \sigma_n)
    = T_n(\sigma_1, -\sigma_2, -\sigma_3, \dots, -\sigma_{n-1}, -\sigma_n) \]
  where the signs of $\sigma_2$, \dots, $\sigma_n$ are flipped on the right hand side.
\end{theorem}

\subsubsection{Recursion via zig-zag numbers}
Fel also observed that the $T_n$ seem to satisfy a recursion relation
related to the tangent/zig-zag numbers $A_{2j+1}$,
satisfying
\[ \sec x + \tan x = \sum_j A_j \cdot x^j/j!. \]
Namely, we have the following theorem.
\begin{theorem}
  [{\cite[Conjecture 2]{ref:felpoly}} or {\cite[Theorem 1.7]{ref:quasimodular}}]
  \label{conj:fel2}
  For $n \ge 1$, we have
  \[ \frac{T_{2n+1}}{T_1^{2n+1}}
    = \sum_{j = 0}^n (-1)^j A_{2j+1} \binom{2n+1}{2j+1} \frac{T_{2n-2j}}{T_1^{2n-2j}}.  \]
\end{theorem}

\subsubsection{Ramanujan's $q$-series}
Recently a new striking appearance of $T_n$ was found in quasimodular forms,
which led to the resolution of the preceding two conjectures.
Consider Ramanujan's $q$-series
\begin{align*}
  U_{2n}(q)
  &\coloneq \frac{1^{2n+1}-3^{2n+1}q+5^{2n+1}q^3-7^{2n+1}q^6+\dots}{1-3q+5q^3-7q^6+\dots} \\
  &= \frac{\sum_k (-1)^k (2k+1)^{2n+1} q^{\binom{k+1}{2}}}
  {\sum_k (-1)^k (2k+1) q^{\binom{k+1}{2}}}
\end{align*}
which is a weight $2n$ quasimodular form \cite{ref:bcss}.
As in \cite[\S3]{ref:trace}, introduce the generating function
\[ \Omega(X) \coloneq \sum_{n \ge 0} U_{2n}(q) \cdot\frac{X^{2n}}{(2n+1)!}
  = \frac{\sinh X}{X} \prod_{j \ge 1} \left( 1 - \frac{4q^j \sinh^2 X}{(1-q^j)^2} \right) \]
and define $Y_n(q)$ according to
\[ \sum_{n \ge 0} Y_n(q) \cdot \frac{X^n}{n!}
  = \exp\left( \frac{\theta(q)^2 X}{2} \right) \Omega\left( \frac X2 \right) \]
where $\theta(q) \coloneq \prod_{k \ge 1}(1-q^{2k})(1+q^{2k-1})^2$
is the weight $1/2$ theta function.
Then we can describe $Y_n$ with the following proposition.

\begin{proposition}
  [{\cite[Theorem 1.1]{ref:quasimodular}}]
  We have a formula
  \[ Y_n(q) = \frac{1}{2^n(n+1)} \sum_{k \ge 0} \binom{n+1}{2k+1} \theta(q)^{2n-4k} U_{2k}(q). \]
\end{proposition}
Each $Y_n$ then has a canonical symmetric function avatar; it is defined by
\[ \wt Y_n(x_1, x_2, \dots)
  \coloneq \frac{1}{2^n (n+1)} \sum_{k \ge 0}
  \binom{n+1}{2k+1} \sigma_1^{n-2k} \Psi(U_{2k}(q)) \]
where $\Psi$ is a symmetric function representation
defined in \cite[equation (1.9)]{ref:quasimodular},
extended with $\Psi(\theta(q)^2) = \sigma_1$.
These avatars turn out to coincide again with Fel's universal symmetric polynomials:
\begin{theorem}
  [{\cite[Theorem 1.3]{ref:quasimodular}}]
  For integers $n, k \ge 1$,
  \[  T_n(x_1, \dots, x_k) = \wt Y_n(x_1, \dots, x_k, 0, 0, 0, \dots). \]
\end{theorem}
In particular, this interpretation was used to prove
both \Cref{conj:fel1} and \Cref{conj:fel2}.

\subsection{Main result}
Given the renewed interest and results in $T_n$, here we close out the original birthplace of the $T_n$
and resolve the remaining \Cref{conj:main} from \cite{ref:fel}.
\begin{theorem}
  \label{thm:main}
  \Cref{conj:main} is true.
  For every $p \ge 0$, we have
  \[ K_p(S) = \sum_{r=0}^p \binom pr T_{p-r}(\sigma) G_r(S)
    + \frac{2^{p+1}}{p+1} T_{p+1}(\delta). \]
\end{theorem}

\begin{remark}
This work is a case study and test case for AxiomProver, an AI tool currently under development,
aimed at end-to-end automated theorem proving in mainstream mathematics.
Starting from a natural-language formulation of Fel’s conjecture,
which includes a self-contained collection of definitions,
AxiomProver generated a Lean/Mathlib statement and a fully verified proof.
Using that formal development as a reference point,
we prepared the exposition in the main text for a mathematical audience, aiming to supply context,
motivation, and a streamlined derivation that can be read independently of the Lean code.
\end{remark}

The rest of the paper is organized as follows.
In \Cref{sec:background}, we provide further background
and define the polynomials $T_n$.
Several examples are explicated in \Cref{sec:examples},
but those are for reference only and may be skipped.
The proof of Theorem~\ref{thm:main} is given in \Cref{sec:proof}.
The Appendix documents the formalization,
clarifies the experimental conditions under which the system was evaluated,
and provides links to the relevant files for interested readers
(Mathematicians not interested in automated theorem proving
can thus safely ignore the appendix.)

\subsection{Acknowledgments}
\label{sec:acknow}
The authors thank Tewodros Amdeberhan and Leonid Fel for discussions related to this paper.
We also thank the anonymous referees for helpful corrections.

\ifmanyauthors
\else
This paper describes a test case for AxiomProver,
an autonomous system that is currently under development.
The project engineering team is
Chris Cummins,
GSM,
Dejan Grubisic,
Leopold Haller,
Letong Hong (principal investigator),
Andranik Kurghinyan,
Kenny Lau,
Hugh Leather,
Aram Markosyan,
Manooshree Patel,
Gaurang Pendharkar,
Vedant Rathi,
Alex Schneidman,
Volker Seeker,
Shubho Sengupta (principal investigator),
Ishan Sinha,
Jimmy Xin,
and Jujian Zhang.
\fi

\section{Background}
Here we recall the relevant nuts and bolts that are required to prove Fel's Conjecture.

\label{sec:background}
\subsection{Syzygy degrees}
On the side of syzygy degrees, we content ourselves to note that $K_p(S)$
was given a complete definition in \Cref{sec:intro}.
Indeed, we have
\begin{itemize}
  \item The numerator $Q_S(z)$ of the Hilbert series $H_S(z)$ is given by \Cref{def:hilbert}.
  \item By combining \eqref{eq:syzygy} and \Cref{def:alternating},
    we may write $\CC_r$ directly in terms of the coefficients of $Q_S$:
    \begin{equation}
      \CC_r(S) = \sum_{n=0}^{\infty} n^r \cdot [z^n] \, (1-Q_S(z))
      \label{eq:Cr}
    \end{equation}
    This definition is somewhat ad hoc and perhaps not the ``morally correct'' one
    because it does not refer to partial Betti numbers $\beta_i$ at all;
    however, for the purposes of proving \Cref{thm:main} it is sufficient.
  \item Finally, \Cref{thm:Cmp_Kp} can be rearranged to say
    \begin{equation}
      K_p(S) = \frac{\CC_{m+p}(S)}{(-1)^m \pi_m \cdot \frac{(m+p)!}{p!}}.
      \label{eq:Kp}
    \end{equation}
\end{itemize}
The rest of the section is dedicated to describing $T_n(x_1, \dots, x_m)$
in terms of $x_i$, $\sigma_k$, and $\delta_k$.

\subsection{Universal symmetric polynomials, in terms of $x_i$}
It remains to give a full definition of the universal symmetric polynomials $T_n$.
Fel's original definition \cite[\S2]{ref:felpoly} begins with
\begin{align*}
  P_n(x_1, \dots, x_m)
  &\coloneq \sum_{1 \le i \le m} x_i^n - \sum_{1 \le i < j \le m} (x_i + x_j)^n
  + \sum_{1 \le i < j < k \le m} (x_i + x_j + x_k)^n \\
  &- \dots + (-1)^m (x_1 + \dots + x_m)^n
\end{align*}
for any integers $m \le n$, and then sets
\begin{equation}
  T_{n-m}(x_1, \dots, x_m)
  \coloneq \frac{P_n(x_1, \dots, x_m)}{x_1 \dots x_m \cdot \frac{(-1)^{m+1} n!}{(n-m)!}}.
  \label{eq:felT}
\end{equation}
We prefer to rewrite \eqref{eq:felT} more compactly using an exponential generating function.
\begin{definition}
  \label{def:A}
  Define the generating function
  \[ A(t) \coloneq \prod_{i=1}^m \frac{e^{x_i t} - 1}{x_i t} \in \QQ[[t]]. \]
  Then we have
  \[ T_n(x_1, \dots, x_m) \coloneq n! \cdot [t^n] \, A(t). \]
\end{definition}
To see that this definition matches \eqref{eq:felT}, it suffices to note that
\begin{align*}
  (-1)^m \prod_{i=1}^m (e^{x_i t} - 1)
  = \prod_{i=1}^m (1 - e^{x_i t})
  &= 1 - \sum_{1 \le i \le m} e^{x_i t} + \sum_{1 \le i < j \le m} e^{(x_i + x_j) t} \\
  &- \sum_{1 \le i < j < k \le m} e^{(x_i + x_j + x_k) t} + \dots \\
  &+ (-1)^m e^{(x_1 + \dots + x_m)t}.
\end{align*}
Hence the coefficient of $t^n$ in the numerator of $A(t)$ in \Cref{def:A}
equals
\[ (-1)^{m+1} n! \cdot P_n(x_1, \dots, x_m) \]
for $n \ge m$.
Adjusting by the denominator of $x_1 \dots x_m \cdot t^m$ gives \eqref{eq:felT}.

\subsection{Universal symmetric polynomials, in terms of $\sigma_k$}
It is clear that $T_n$ is symmetric in the $x_i$'s and of degree $n$.
Hence we will rewrite it in the power sum basis $\sigma_k \coloneq \sum_{i=1}^m x_i^k$.
For example, rather than writing
\begin{align*}
  T_1(x_1, \dots, x_m) &= \frac{x_1 + \dots + x_m}{2} \\
  T_2(x_1, \dots, x_m) &= \frac{3(x_1 + \dots + x_m)^2 + (x_1^2 + \dots + x_m^2)}{12} \\
  &\vdotswithin=
\end{align*}
we can write the more compact
\begin{align*}
  T_1(\sigma) &= \frac{\sigma_1}{2} \\
  T_2(\sigma) &= \frac{3\sigma_1^2 + \sigma_2}{12}
\end{align*}
viewing $T_i$ as a polynomial in $\sigma = (\sigma_1, \sigma_2, \dots, \sigma_i)$.
This has the advantage that it suppresses the dependence on $m$ from the definition.
Further examples for $n \le 7$ are recorded in \Cref{sec:example_Tn}.

\subsection{Universal symmetric polynomials, in terms of $\delta_k$}
Fel's conjecture involves replacing $T_n(\sigma)$ with $T_n(\delta)$,
by replacing each $\sigma_k$ with $\delta_k \coloneq \frac{\sigma_k-1}{2^k}$.
We'll also write an exponential generating function for $T_n(\delta)$:
\begin{definition}
  \label{def:B}
  Retaining $A(t)$ from \Cref{def:A}, define
  \[ B(t) \coloneq \frac{t}{e^t-1} A(t) \in \QQ[[t]]. \]
  Then we have
  \[ T_n(\delta) \coloneq \frac{n!}{2^n} \cdot [t^n] \, B(t). \]
\end{definition}
To see why $B(t)$ works,
introduce the symbols $y_1$, \dots, $y_m$ satisfying
\begin{equation}
  x_1^k + x_2^k + \dots + x_m^k = y_1^k + \dots + y_m^k + 1 \qquad k \ge 1.
  \label{eq:rig}
\end{equation}
(Strictly speaking, we should really use infinitely many variables,
so that $\sigma_1$, $\sigma_2$, \dots\ are all independent.
Alternatively one may fix $n$, choose $m > n$, and work modulo $t^{m+1}$.)
Then applying \Cref{def:A} and \eqref{eq:rig} means
\begin{align*}
  \prod_{i=1}^m \frac{e^{y_i t} - 1}{y_i t}
  &= \sum_{n=0}^{\infty} \frac{T_n(y_1, \dots, y_m)}{n!} t^n
  = \sum_{n=0}^{\infty} \frac{T_n(\sigma-1)}{n!} t^n \\
  \frac{e^t-1}{t} \cdot \prod_{i=1}^m \frac{e^{y_i t} - 1}{y_i t}
  &= \sum_{n=0}^\infty \frac{T_n(y_1, \dots, y_m, 1)}{n!} t^n
  = \sum \frac{T_n(\sigma)}{n!} t^n = A(t).
\end{align*}
Here $\sigma-1$ stands for $(\sigma_1-1, \sigma_2-1, \dots)$.
Hence, we conclude that
\[ \sum_{n \ge 0} \frac{T_n(\sigma-1)}{n!} t^n = \frac{t}{e^t-1} A(t) = B(t). \]
Finally, it is clear from homogeneity that
$T_n(\delta) = \frac{1}{2^n} T_n(\sigma-1)$, as needed.

\section{Worked examples}
\label{sec:examples}
For concreteness, we give three examples of the conjecture.
The first subsection shows, for reference,
the formula for $T_n(\sigma)$ for all $0 \le n \le 7$, copied from \cite{ref:fel}.
Then we work out by hand the values of $\Delta$, $G_r$, $\sigma_k$, $\delta_k$,
$\CC_n$, and $K_p$.
These explicit examples are only for reference,
and can be skipped with no loss of continuity.
Furthermore, we offer suitable Sage code in the GitHub link in the
Appendix that was used to generate these examples.

\subsection{The universal symmetric polynomials $T_n$.}
\label{sec:example_Tn}
For concreteness, the first several $T_n$ in terms of $\sigma_k$
(taken from \cite[p.~177]{ref:fel} or \cite[Appendix 3]{ref:fel}) are:
\begin{align*}
  T_0(\sigma) &= 1 \\
  T_1(\sigma) &= \frac{\sigma_1}{2} \\
  T_2(\sigma) &= \frac{3\sigma_1^2+\sigma_2}{12} \\
  T_3(\sigma) &= \sigma_1 \cdot \frac{\sigma_1^2 + \sigma_2}{8} \\
  T_4(\sigma) &= \frac{15\sigma_1^4 + 30\sigma_1^2\sigma_2 + 5\sigma_2^2 - 2\sigma_4}{240} \\
  T_5(\sigma) &= \sigma_1 \cdot \frac{3\sigma_1^4 + 10\sigma_1^2\sigma_2 + 5\sigma_2^2
    - 2\sigma_4}{96} \\
  T_6(\sigma) &= \frac{63\sigma_1^6 + 315\sigma_1^4\sigma_2 + 315\sigma_1^2\sigma_2^2 - 126\sigma_1^2\sigma_4
    + 35\sigma_2^3 - 42\sigma_2\sigma_4 + 16\sigma_6}{4032} \\
  T_7(\sigma) &= \sigma_1 \cdot \frac{9\sigma_1^6 + 63\sigma_1^4\sigma_2 + 105\sigma_1^2\sigma_2^2 - 42\sigma_1^2\sigma_4 + 35\sigma_2^3
    - 42\sigma_2\sigma_4 + 16\sigma_6}{1152}.
\end{align*}

\subsection{Worked example $S = \left\langle 3,5 \right\rangle$}
The semigroup $S = \left\langle d_1, d_2 \right\rangle = \left\langle 3,5 \right\rangle$
has gap set
\[ \Delta = \{1,2,4,7\}. \]
Hence, we have
\begin{align*}
  G_r &= 1 + 2^r + 4^r + 7^r \\
  \sigma_k &= 3^k + 5^k \\
  \delta_k &= \frac{3^k + 5^k - 1}{2^k}.
\end{align*}
We can compute the Hilbert series as
\[ H_S(z) = \frac{1}{1-z} - (z+z^2+z^4+z^7) = \frac{1-z^{15}}{(1-z^3)(1-z^5)}. \]
This means the numerator of the Hilbert series is $Q_S(z) = 1-z^{15}$ and we obtain
\[ \CC_n(S) = 15^n. \]
Thus, we have
\begin{align*}
  K_p(S) &= \frac{\CC_{2+p}(S)}{(-1)^2 \cdot 3 \cdot 5 \cdot (p+1)(p+2)} \\
  &= \frac{15^{p+2}}{15(p+1)(p+2)} = \frac{15^{p+1}}{(p+1)(p+2)}.
\end{align*}

\begin{remark}
  In general when $m = 2$, the ``obvious'' patterns above hold:
  given $\gcd(d_1, d_2) = 1$
  we generally have $Q_S(z) = 1 - z^{d_1 d_2}$
  and \[ K_p(S) = \frac{(d_1d_2)^{p+1}}{(p+1)(p+2)}. \]
  We note this also follows from \cite[equation (63)]{ref:fel}.
\end{remark}

\subsection{Worked example $S = \left\langle 4,5,6 \right\rangle$}
The semigroup $S = \left\langle d_1, d_2, d_3 \right\rangle = \left\langle 4,5,6 \right\rangle$
has gap set
\[ \Delta = \{  1, 2, 3, 7 \}. \]
The corresponding $G_r$, $\sigma_k$, $\delta_k$ are
\begin{align*}
  G_r &= 1 + 2^r + 3^r + 7^r \\
  \sigma_k &= 4^k + 5^k + 6^k \\
  \delta_k &= \frac{4^k + 5^k + 6^k - 1}{2^k}.
\end{align*}
We can compute the Hilbert series as
\[ H_S(z) = \frac{1}{1-z} - (z+z^2+z^3+z^7)
  = \frac{1 - z^{10} - z^{12} + z^{22}}{(1-z^4)(1-z^5)(1-z^6)}. \]
This means the numerator of the Hilbert series is $Q_S(z) = 1-z^{10}-z^{12}+z^{22}$ and we obtain
\[ \CC_n(S) = 10^n + 12^n - 22^n. \]
Thus, we have
\begin{align*}
  K_p(S) &= \frac{\CC_{3+p}(S)}{(-1)^3 \cdot 4 \cdot 5 \cdot 6 \cdot (p+1)(p+2)(p+3)} \\
  &= \frac{22^{p+3} - 10^{p+3} - 12^{p+3}}{120(p+1)(p+2)(p+3)}.
\end{align*}
We note this result also follows from \cite[equation (63)]{ref:fel},
since $S$ is a symmetric complete intersection $3$-generated semigroup,
with $Q_S(z) = (1-z^{10})(1-z^{12})$.

\subsection{Worked example $S = \left\langle 5,6,8,9 \right\rangle$}
The semigroup $S = \left\langle d_1, d_2, d_3, d_4 \right\rangle = \left\langle 5,6,8,9 \right\rangle$
has a gap set
\[ \Delta = \left\{ 1,2,3,4,7 \right\}. \]
The corresponding $G_r$, $\sigma_k$, $\delta_k$ are
\begin{align*}
  G_r &= 1 + 2^r + 3^r + 4^r + 7^r \\
  \sigma_k &= 5^k + 6^k + 8^k + 9^k \\
  \delta_k &= \frac{5^k + 6^k + 8^k + 9^k - 1}{2^k}.
\end{align*}
Computing the Hilbert series is more involved:
\[
  H_S(z) = \frac{1}{1-z} - (z+z^2+z^3+z^4+z^7)
  = \frac{Q_S(z)}{(1-z^5)(1-z^6)(1-z^8)(1-z^9)}
\]
where
\begin{align*}
  Q_S(z) &= 1 - z^{14} - z^{15} - z^{16} - z^{17} - 2 z^{18} + z^{22} + 2 z^{23} \\
    &\qquad+ z^{24} + z^{25} + 2 z^{26} + z^{27} - z^{31} - z^{32} - z^{35}.
\end{align*}
We obtain
\begin{align*}
  \CC_n(S) &= 14^n + 15^n + 16^n + 17^n + 2 \cdot 18^n \\
  &\qquad-22^n - 2 \cdot 23^n - 24^n - 25^n - 2 \cdot 26^n - 27^n \\
  &\qquad+31^n + 32^n + 35^n \\
  K_p(S) &= \frac{\CC_{p+4}(S)}{2160(p+1)(p+2)(p+3)(p+4)}.
\end{align*}

\section{Proof of \Cref{thm:main}}
\label{sec:proof}

Fix the semigroup $S = \left\langle d_1, \dots, d_m \right\rangle$ above.
The main idea of the proof is to convert all the data into exponential generating functions.
Then purely algebra manipulation will be enough to imply the desired result.
This has already been done for $A(t)$ and $B(t)$,
but the $G_r(S)$ and $\CC_r(S)$ have not been translated yet.
In general we will use the letter $t$ for the variable
in an exponential generating function and $z$ for an ordinary generating function,
so we can summarize the strategy as replacing $z$ with $e^t$ whenever possible.

For example, here is how to work with $\CC_n(S)$ with an exponential generating function.
\begin{lemma}
  \label{lem:C_n_eq_factorial_times_coeff}
  The exponential generating function for $\CC_n(S)$ is $1-Q_S$. That is,
  \[ \CC_n(S) = n! \cdot [t^n] (1 - Q_S(e^t)). \]
\end{lemma}
\begin{proof}
  By comparison with \eqref{eq:Cr}.
  Suppose we write $1 - Q_S(z) = \sum_{k \ge 0} c_k z^k$, so that
  \[ \CC_r(S) = \sum_{k = 0}^\infty k^r \cdot c_k. \]
  In that case, we have
  \begin{align*}
    1 - Q_S(e^t) &= \sum_{k \ge 0} c_k e^{kt} \\
    &= \sum_{k \ge 0} c_k \cdot \sum_{n \ge 0} k^n \cdot \frac{t^n}{n!} \\
    &= \sum_{n \ge 0} \left( \sum_{k \ge 0} c_k \cdot k^n \right) \cdot \frac{t^n}{n!} \\
    &= \sum_{n \ge 0} \CC_r(S) \cdot \frac{t^n}{n!}. \qedhere
  \end{align*}
\end{proof}

Similarly, we can convert $G_r(S)$ into an exponential generating function as follows.
\begin{definition}
  The \emph{gap polynomial} $\Phi_S(z)$ is defined as
  \[ \Phi_S(z) = \frac{1}{1-z} - H_S(z) = \sum_{g \in \Delta} z^g. \]
\end{definition}
\begin{lemma}
  \label{lem:Phi_at_exp_eq_gap_egf}
  The exponential generating function for $G_n(S)$ is given by $\Phi_S(e^t)$, that is
  \[ \Phi_S(e^t) = \sum_{n=0}^{\infty} G_n(S) \cdot \frac{t^n}{n!}. \]
\end{lemma}
\begin{proof}
  By direct computation, we have
  \begin{align*}
    \Phi_S(e^t) &= \sum_{g \in \Delta} (e^t)^g
    = \sum_{g \in \Delta} e^{tg} \\
    &= \sum_{g \in \Delta} \sum_{n \ge 0} \frac{(gt)^n}{n!}
    = \sum_{n \ge 0} \sum_{g \in \Delta} \frac{(gt)^n}{n!}
    = \sum_{n \ge 0} G_n(S) \cdot \frac{t^n}{n!} \qedhere.
  \end{align*}
\end{proof}

\begin{definition}
  The \emph{product polynomial} $P_S(z)$, defined by
  \[ P_S(z) \coloneq \prod_{i=1}^m (1-z^{d_i}) \]
  is the denominator of our definition of the Hilbert series $H_S(z)$.
  Hence, we have $H_S(z) = \frac{Q_S(z)}{P_S(z)}$.
\end{definition}

\begin{lemma}
  \label{lem:poly_at_exp_P_as_series}
  We have
  \[ P_S(e^t) = (-1)^m \pi_m \cdot t^m \cdot A(t). \]
\end{lemma}
\begin{proof}
  By definition
  \[ P_S(e^t) = \prod_{i=1}^m \left( 1 - e^{d_i t} \right) \]
  so this follows by comparison to \Cref{def:A}.
\end{proof}

\begin{lemma}
  \label{lem:P_div_at_exp_series}
  We have
  \[ \frac{P_S(e^t)}{1-e^t} = (-1)^{m+1} \pi_m \cdot t^{m-1} \cdot B(t). \]
\end{lemma}
\begin{proof}
  In the same way as \Cref{lem:poly_at_exp_P_as_series}
  this follows by comparison to \Cref{def:B}.
\end{proof}

\begin{lemma}
  \label{lem:one_minus_Q_at_exp_coeff}
  We have
  \[ 1 - Q_S(e^t) = 1 + (-1)^m \pi_m \cdot
    \left( t^{m-1} \cdot B(t) + t^m \cdot A(t) \cdot \Phi_S(e^t) \right). \]
\end{lemma}
\begin{proof}
  We start from the identity
  \[ \Phi_S(z) = \frac{1}{1-z} - H_S(z) = \frac{1}{1-z} - \frac{Q_S(z)}{P_S(z)} \]
  so that
  \[ Q_S(z) = \frac{P_S(z)}{1-z} - \Phi_S(z) P_S(z)  \]
  and
  \[ 1 - Q_S(z) = 1 - \frac{P_S(z)}{1-z} + \Phi_S(z) P_S(z). \]
  Substituting $z = e^t$ in \Cref{lem:P_div_at_exp_series},
  and \Cref{lem:poly_at_exp_P_as_series} to rewrite $P_S(e^t)$ in terms of $A(t)$ and $B(t)$
  completes the proof.
\end{proof}

Now choose $n = m + p$, for any $p \ge 0$,
and proceed to extract the $t^{m+p}$ coefficient of \Cref{lem:one_minus_Q_at_exp_coeff}.
On the left-hand side, using \Cref{lem:C_n_eq_factorial_times_coeff} we get
\begin{equation}
  \label{eq:final_C}
  [t^{m+p}] (1 - Q_S(e^t)) = \frac{\CC_{m+p}(S)}{(m+p)!}.
\end{equation}
We compare this with the following two calculations.

\begin{lemma}
  \label{lem:B_term_contribution}
  We have
  \[ p! \cdot [t^{p+1}] B(t) = \frac{2^{p+1}}{p+1} T_{p+1}(\delta). \]
\end{lemma}
\begin{proof}
  From \Cref{def:B}, we have
  \[ [t^{p+1}] B(t) = \frac{2^{p+1}}{(p+1)!} T_{p+1}(\delta) \]
  which gives the desired conclusion by multiplying by $p!$.
\end{proof}

\begin{lemma}
  \label{lem:gap_term_contribution}
  We have
  \[ p! \cdot [t^p] A(t) \cdot \Phi_S(e^t)
    = \sum_{r=0}^p \binom pr T_{p-r}(\sigma) \cdot G_r(S). \]
\end{lemma}
\begin{proof}
  Compute
  \begin{align*}
    [t^p] A(t) \cdot \Phi_S(e^t)
    &= \sum_{r=0}^p [t^{p-r}] A(t) \cdot [t^r] \Phi_S(e^t) \\
    &= \sum_{r=0}^p \frac{T_{p-r}(\sigma)}{(p-r)!} \cdot \frac{G_r(S)}{r!} \\
    &= \frac{1}{p!} \sum_{r=0}^p \binom pr T_{p-r}(\sigma) G_r(S). \qedhere
  \end{align*}
\end{proof}
Hence, if we compare the $[t^{m+p}]$ coefficients
by combining \eqref{eq:final_C}
with the sum of \Cref{lem:B_term_contribution} and \Cref{lem:gap_term_contribution}
in \Cref{lem:one_minus_Q_at_exp_coeff},
we obtain that
\begin{equation}
  \label{eq:final}
  \frac{\CC_{m+p}(S)}{(m+p)!} = \frac{(-1)^m}{p!} \pi_m \cdot \left(
      \frac{2^{p+1}}{p+1} T_{p+1}(\delta)
      + \sum_{r=0}^p \binom pr T_{p-r}(\sigma) G_r(S) \right).
\end{equation}
Dividing out by $\frac{(-1)^m}{p!} \pi_m$ gives \Cref{conj:main}, as desired,
completing the proof of Theorem~\ref{thm:main}.
\appendix
\section*{Appendix}
\setcounter{section}{1}
\label{sec:appendix}
\smallskip

AxiomProver is an AI system for mathematical research via formal proof that is currently under development.
As an early test case in this effort, we present this case study,
treating \Cref{conj:main} as an end-to-end formalization target.
We chose it because we believed that its proof could be driven largely
without requiring a full formal development of numerical semigroups,
minimal resolutions, symmetric functions, and related machinery,
placing it within reach of today’s Mathlib.
We also chose the problem because it has gained renewed relevance given the
recent reemergence of the polynomials $T_n$.

This paper confirms that expectation: the formal statement and proof
are fully formalized in Lean/Mathlib (see \cite{Lean,Mathlib2020}) and were produced
by AxiomProver from a natural-language statement of \Cref{conj:main}.
The following section makes
precise what ``produced’’ means, and describes our process.

\subsection*{Process}
The formal proofs provided in this work were developed and verified using Lean \textbf{4.26.0}.
Compatibility with earlier or later versions is not guaranteed
due to the evolving nature of the Lean 4 compiler and its core libraries.
The relevant files are all posted in the following repository:
\begin{center}
  \url{https://github.com/AxiomMath/fel-polynomial}
\end{center}
The input files were
\begin{itemize}
  \item \texttt{Fel\_Conjecture.tex}, a natural-language statement of the problem;
  \item a \texttt{task.md} that contains the single line
  \begin{quote}
    \slshape
    1. State and prove Fel's conjecture in Lean.
  \end{quote}
  \item a configuration file \texttt{.environment} that contains the single line
  \begin{quote}
    \slshape
    lean-4.26.0
  \end{quote}
  which specifies to AxiomProver which version of Lean should be used.
\end{itemize}
Given these three input files,
AxiomProver autonomously provided the following output files:
\begin{itemize}
  \item \texttt{problem.lean}, a Lean 4.26.0 formalization of the problem statement; and
  \item \texttt{solution.lean}, a complete Lean 4.26.0 formalization of the proof.
\end{itemize}

The repository also contains an ancillary file \texttt{examples.sage} that
provides the examples in the earlier Section~\ref{sec:examples}.
This file was written by hand and is unrelated to the formalization process
(in particular, it was not provided to AxiomProver as part of the input).

After AxiomProver generated a solution, the human authors wrote this paper
(without the use of AI) for human readers.
Indeed, a research paper is a narrative designed to communicate ideas to humans,
whereas a Lean files are designed to satisfy a computer kernel.

\subsection*{Further commentary}
We offer further commentary and perspective for these results.

\smallskip
\noindent
(1) This work represents a case where an AI system (i.e.\ AxiomProver)
  is given a specified and self-contained natural language problem statement,
  which it in turn solves and verifies in Lean.
  We have already known for a while that AI can often
  solve and formalize competition problems from Putnam and IMO.
  This paper is an early example of the end-to-end formalization of research mathematics.

\smallskip
\noindent
(2) We specifically selected an open problem that we believed was
  within reach of the current Mathlib repository.
  We didn't want to worry about gaps in Mathlib preventing the formalization.

\printbibliography[title=References]

\end{document}